\documentclass{amsart}

\usepackage[top=1.1in,bottom=1.1in,left=1.2in,right=1.2in]{geometry}
\usepackage{amsmath,amssymb,amsthm,amsfonts,enumerate, mathrsfs, mathtools, appendix}
\usepackage{stmaryrd}
\usepackage{esint}
\usepackage{xcolor}
\usepackage{graphicx}
\usepackage{tikz} 
\usetikzlibrary{patterns,positioning,arrows,shapes,trees}
\usepackage[%colorlinks=false,   
colorlinks=true,
linkcolor=blue,
filecolor=blue,
urlcolor=red,
citecolor=magenta]{hyperref}

\usepackage{color}
\definecolor{MyDarkBlue}{RGB}{54,117,23}
\definecolor{MyDarkBlue}{cmyk}{0.8,0.3,0.8,0.4}
\definecolor{yellow}{rgb}{0.99,0.99,0.70}
\definecolor{white}{rgb}{1.0,1.0,1.0}
\definecolor{black}{rgb}{0.00,0.00,0.00}

\makeatletter
\newcommand{\mylabel}[2]{#2\def\@currentlabel{#2}\label{#1}}
\makeatother

\newcounter{counterConstant}

\numberwithin{equation}{section}

\newtheorem{theorem}{Theorem}[section]
\newtheorem{lemma}[theorem]{Lemma}
\newtheorem{remark}[theorem]{Remark}

\newtheorem{corollary}[theorem]{Corollary}

\def\cF{{\mathcal F}}

\def\mI{{\mathbb I}}

\def\mN{{\mathbb N}}

\def\mR{{\mathbb R}}
\def\mS{{\mathbb S}}

\def\bE{{\mathbf E}}

\def\bP{{\mathbf P}}

\def\sS{{\mathscr S}}

\def\l{\left}
\def\r{\right}
\def\<{\langle}
\def\>{\rangle}

\def\geq{\geqslant}
\def\leq{\leqslant}

\def\1{{\mathbf{1}}}

\def\d{\text{\rm{d}}}
\def\e{\mathrm{e}}
\def\eps{\varepsilon}
\def\Re{{\mathrm{Re}}}

\allowdisplaybreaks

\begin{document}

\title{Euler-Maruyama scheme for SDE driven by L\'evy process with H\"older drift}

\author{Yanfang Li and Guohuan Zhao} 

\address{Institute of Applied Mathematics, Academy of Mathematics and Systems Science, CAS, Beijing, 100190, China}
\email{liyanfang@amss.ac.cn}

\address{Institute of Applied Mathematics, Academy of Mathematics and Systems Science, CAS, Beijing, 100190, China}
\email{gzhao@amss.ac.cn}

\thanks{Research Guohuan is supported by the National Natural Science Foundation of China (No. 12288201); Research Yanfang is supported by the China Postdoctoral Science Foundation (No. 2022M723328)} 

\begin{abstract} 
This study focuses on approximating solutions to SDEs driven by L\'evy processes with H\"older continuous drifts using the Euler-Maruyama scheme. We derive the $L^p$-error for a broad range of driven noises, including all nondegenerate $\alpha$-stable processes ($0<\alpha<2$).
\end{abstract}

\maketitle
%\tableofcontents

\bigskip
\noindent 
\textbf{Keywords}: Stochastic differential equation, L\'evy process, Euler-Maruyama scheme 

\noindent
  {\bf AMS 2020 Mathematics Subject Classification: Primary 39A50;  Secondary 41A25, 60J76} 
\section{Introduction}
Let $Z$ be a pure jump L\'evy process whose characteristic exponent is given by 
\[
\psi(\xi) = \log \bE \e^{i \xi\cdot Z_1} = \int_{\mR^d} \l(\e^{i\xi\cdot z}-1-i \xi\cdot z \1_{B_1}(z)\r) \nu(\d z), 
\]
where $\nu$ is the intensity measure satisfying $\int_{\mR^d} (1\wedge |z|^2)\nu(\d z)<\infty$. Consider the following stochastic differential equation (SDE) driven by $Z$: 
\begin{equation}\label{Eq-SDE}
X_t=X_0+\int_0^t b(X_s) \d s + Z_t,  
\end{equation}
as well as its Euler-Maruyama scheme 
\begin{equation}\label{Eq-EM}
X_t^n= X_0^n+ \int_0^t b(X^n_{k_n(s)}) \d s + Z_t,   
\end{equation}
where $k_n(t):= [nt]/n$.

The main purpose of this paper is to study the strong convergence rate of the Euler-Maruyama approximation for \eqref{Eq-SDE}, where $Z$ belongs to a wide class of L\'evy processes. The main result is formulated as follows: 
\begin{theorem}\label{Thm}
Assume that there are constants $c_0>0$, $\alpha\in (0,2)$ and $ M>0$ such that 
\begin{equation}\label{Eq-aspt1}
\Re(-\psi(\xi)) \geq c_0 |\xi|^\alpha, \ \mbox{ for all } |\xi|\geq M,  
\end{equation} 
and $b\in C^\beta$ with $\beta\in (1-\alpha/2, 1)$.  Then for each $p>0$, there is a constant C depending on $d, c_0, M, \alpha, \beta,  p, \|b\|_{\beta}$ such that 
\begin{equation}\label{Eq-main}
\bE \sup_{t\in [0,1]} |X_t^n-X_t|^p \leq C \l[ \bE|X^n_0-X_0|^p+\bE \l(1\wedge |Z_{1/n}|^{p\beta}\r)\r].
\end{equation}
\end{theorem}

\begin{remark}
    One sufficient condition for $Z$ to satisfy \eqref{Eq-aspt1} is 
    \begin{equation}\label{eq-nondege}
    \int_{|z|\leq \rho} |\eta\cdot z|^2 \nu(\d z) \geq c \rho^{2-\alpha},  \  \forall \eta\in \mS^{d-1}, \rho \in (0,\rho_0], 
    \end{equation}
    where $c>0$ and $\rho_0>0$ are two positive constants. The reason is: using $1-\cos t \geq t^2/3 \ (t\in [-1, 1])$, we have 
    \begin{align*}
    \int_{\mR^d} \l( 1-\cos(\xi\cdot z) \r) \nu(\d z)
    \geq& c \int_{|z|\leq |\xi|^{-1}} |z\cdot \xi|^2 \nu(\d z) \\
    \geq& c |\xi|^2 \int_{|z|\leq |\xi|^{-1}} |\hat{\xi} \cdot z|^2 \nu(\d z ) \overset{\eqref{eq-nondege}}{\geq} c_0 |\xi|^\alpha, \quad \forall \, |\xi| \gg 1. 
    \end{align*}
\end{remark}

The Euler-Maruyama approximation of stochastic differential equations (SDEs) is a well-established field of research in probability theory and numerical analysis, with a vast body of literature dedicated to it. A notable phenomenon in this area is the regularization of the noise for schemes with irregular drift. For instance, Gy\"{o}ngy-Krylov \cite{gyongy1996existence} established the convergence (without an explicit rate) of the Euler-Maruyama scheme when the driven noise is the Brownian motion and the drift coefficient satisfies certain integrability conditions. Recently, researchers have imposed $\beta$-H\"older type conditions on the modulus of continuity of drifts in \cite{long2017strong}, \cite{bao2019convergence}, and \cite{suo2022weak} (with the latter two works discussing more general cases). Although the index $\beta$ can be arbitrarily small, the drawback is that the convergence rates obtained become increasingly worse as $\beta$ approaches zero. When the driven noise is an $\alpha$-stable process with $\alpha\in (0,2)$, and the drift coefficient is only $\beta$-H\"older continuous with $\beta>1-\alpha/2$, the strong well-posedness of \eqref{Eq-SDE} has been studied in \cite{tanaka1974perturbation}, \cite{priola2012pathwise}, \cite{priola2015stochastic},\cite{chen2018stochastic}, and \cite{chen2021supercritical}. However, only for $\alpha\in [1,2)$, the rate of strong convergence for the Euler-Maruyama approximation of SDE \eqref{Eq-SDE} has been studied in \cite{menoukeu2017strong}, \cite{mikulevivcius2018rate}, \cite{huang2018euler}, and \cite{kuhn2019strong}. Notably, all of the convergence rates obtained in these works depend on the regularity of the drift coefficient, and they become increasingly worse as $\beta$ approaches $1-\alpha/2$. 

Our contribution is to relax some constraints on noise in previous studies. Specifically, we address a scenario where the intensity measure $\nu$ is singular with respect to the Lebesgue measure and the parameter $\alpha$ lies in the whole range $(0,2)$. Our approach is quite straightforward. In section \ref{Sec-auxiliary},we consider the following resolvent equation that corresponds to \eqref{Eq-SDE}
\begin{equation}\label{Eq-PDE}
\lambda u-Lu-b\cdot\nabla u= f, 
\end{equation}
where $L$ is the infinitesimal generator of $Z$, i.e. 
\[
  L u(x) =\int_{\mR^d}\l( f(x+z)-f(x)-\nabla f(x)\cdot z \1_{B_1}\r)~\nu(\d z). 
\]
Using the similar line of proof from the second named author's previous work \cite[Theorem 1.1]{zhao2021regularity} (see also \cite{chen2021supercritical}), we arrive at the primary auxiliary analytic result, Theorem \ref{Thm-Regularity}, which establishes a good regularity estimate for solutions to \eqref{Eq-PDE}. Then in section \ref{Sec-main}, by re-expressing the drift term in equations \eqref{Eq-SDE} and \eqref{Eq-EM} in a form that facilitates comparison between the two, we prove our main result. 

We close this section by mentioning the remarkable contributions of recent  works, such as \cite{dareiotis2020regularisation}, \cite{butkovsky2021approximation} and \cite{bencheikh2022convergence}, which have shown that an almost $1/2$ rate of convergence holds for all H\"older (or Dini) continuous coefficients, when the driven noises are Brownian motions. These results are further supported by related works such as \cite{leobacher2017strong} and \cite{muller2020performance}. However, to the best of our knowledge, there is currently no literature that investigates whether the convergence rate is insensitive to the regularity of $b$ when $Z$ is modeled by an $\alpha$-stable process. This issue is beyond the scope of this short note, so we will investigate it in our future work.
\section{Auxiliary Results}\label{Sec-auxiliary}
In this section, we formulate some results that will be used in the proof of Theorem \ref{Thm}. Before that let us introduce some notions and recall very basic facts from Littlewood-Paley theory. Let $\sS(\mR^d)$ be the Schwartz space of all rapidly decreasing functions, and $\sS'(\mR^d)$ the dual space of $\sS(\mR^d)$ 
called Schwartz generalized function (or tempered distribution) space. Denote the Fourier transform of $f\in \sS'(\mR^d)$ by $\cF f$ or  $\hat f$. 

Let $\chi:\mR^d\to[0,1]$ be a smooth radial function so that $\chi|_{B_{3/4}}=1$ and $\chi|_{B_1^c}=0$. Define
$$
\varphi(\xi):=\chi(\xi)-\chi(2\xi).
$$
The dyadic block operator 
$\Delta_j$ is defined by
$$
\Delta_j f:=
\left\{
\begin{array}{ll}
\cF^{-1}(\chi(2\cdot) \cF f), & j=-1, \\
\cF^{-1}(\varphi(2^{-j}\cdot) \cF f),& j\geq 0.
\end{array}
\right.
$$
For $s\in\mR$ and $p\in[1,\infty]$, the Besov space $B^s_{p,p}$ is defined as the set of all $f\in\sS'(\mR^d)$ with
\[
\|f\|_{B^s_{p,p}}:=\1_{\{p<\infty\}}\left(\sum_{j\geq -1}2^{jsp}\|\Delta_j f\|_p^p\right)^{1/p}+\1_{\{p=\infty\}}\left(\sup_{j\geq -1}2^{js}\|\Delta_j f\|_p\right)<\infty.
\]

For each $s>0$ with $s\notin \mN$, and $p\in [1,\infty)$, the H\"older space and Sobolev-Slobodeckij space are defined by 
\[
\|f\|_{C^s}= \sum_{0\leq k\leq [s]}\|\nabla^kf\|_{L^\infty}+\sup_{x\neq y} \frac{|\nabla^{[s]}f(x)-\nabla^{[s]}f(y)|}{|x-y|^{s-[s]}}
\]
and 
\[
\|f\|_{W^s_p}:= \sum_{0\leq k\leq [s]}\|\nabla^kf\|_{L^p} + \l( \int\!\!\!\int_{\mR^d\times\mR^d} \frac{|\nabla^{[s]}f(x)-\nabla^{[s]}f(y)|}{|x-y|^{d+(s-[s])p}} \r)^{1/p}, 
\]
respectively. We have the following relations for the above functional spaces: 
\begin{equation*}%\label{Eq-relation1}
    B^s_{p,p}=W^s_p, \quad B^{s}_{\infty,\infty}=C^s 
\end{equation*}
(see for instance \cite{triebel1992theory}). 

In the following, we present some analytical results that are required to prove our main theorem. Firstly, we introduce a novel form of Bernstein's inequality, as presented in \cite{chen2007new}.
\begin{lemma}\label{Le-Bernstein1}
 For any $2\leq p<\infty$, $j\geq 0$ and $\alpha\in(0,2)$, there is a constant $c>0$ such that for all $f\in\sS'(\mR^d)$, 
\begin{align}\label{Eq-NewBernstein1}
\int_{\mR^d}\Big|(-\Delta)^{\alpha/4}|\Delta_j f|^{p/2}\Big|^2 \geq c 2^{\alpha j}\|\Delta_j f\|_p^p. \end{align}
\end{lemma}
Secondly, following \cite{chen2021supercritical}, we have the following crucial lemma. 
\begin{lemma}\label{Le-Bernstein2}
Suppose $\Re (-\psi (\xi))\geq c_0|\xi|^\alpha$ for some $c_0>0$ and $|\xi|\geq M>0$. Then for any 
$p>2$, there are constants 
$c_p=c(c_0,p)>0$ and $j_0=j_0(c_0, M)$ such that for all $j\geq j_0$, it holds that 
\begin{align}\label{Eq-Bernstein2}
 \int_{\mR^d}|\Delta_j f|^{p-2} ( \Delta_j f) L  \Delta_j f \leq -c_p 2^{\alpha j}\|\Delta_j f\|_p^p,
  \end{align}
and for all $j=-1, 0, 1,\cdots$, 
\begin{equation}\label{Eq-positive}
 \int_{\mR^d}|\Delta_{j} f|^{p-2} (\Delta_{j} f) L  \Delta_{j} f\leq 0.
 \end{equation}
\end{lemma}

\begin{proof} 
For $p\geq 2$, by the elementary inequality 
 $|r|^{p/2}-1\geq \frac{p}{2}(r-1)$ for $r\in\mR$, 
we have
$$
 |a|^{p/2}-|b|^{p/2}\geq \tfrac{p}{2}(a-b)b|b|^{p/2-2} ,\  \ a,b\in\mR.
 $$
Letting $g$ be a smooth function, by definition we have
\begin{align*}
L  |g|^{p/2}(x)&=\int_{\mR^d}\Big(|g(x+z)|^{p/2}-|g(x)|^{p/2}-\1_{B_1}(z)z\cdot\nabla |g(x)|^{p/2}\Big)\nu(\d  z)\\
 &\geq\frac{p}{2}|g(x)|^{p/2-2}g(x)\int_{\mR^d}\Big(g(x+z)-g(x)-\1_{B_1}(z)z\cdot\nabla g(x)\Big)\nu(\d  z)\\
 &=\frac{p}{2}|g(x)|^{p/2-2}g(x)L  g(x).
 \end{align*}
 Multiplying both sides by $|g|^{p/2}$ and then integrating in $x$ over $\mR^d$, by Plancherel's formula, we obtain
\begin{align*}
\int_{\mR^d}|g|^{p-2}g\,L  g &\leq\frac{2}{p}\int_{\mR^d}|g|^{p/2}L  |g|^{p/2} 
=\frac{2}{p}\int_{\mR^d}|\cF(|g|^{p/2})(\xi)|^2\psi(\xi)\d \xi\\
&=\frac{2}{p}\int_{\mR^d}|\cF(|g|^{p/2})(\xi)|^2\mathrm{Re}(\psi (\xi))\d \xi\leq 0, 
\end{align*}
which implies \eqref{Eq-positive}. Moreover, by our assumption on $\psi$, we have $\mathrm{Re}(-\psi (\xi))\geq c_0 |\xi|^\alpha-M^\alpha$ (for all $\xi\in \mR^d$). Thus, 
\begin{align*}
\int_{\mR^d}|g|^{p-2}g\,L  g &\leq -\frac{2c_0}{p}\int_{\mR^d}|\cF(|g|^{p/2})(\xi)|^2 (|\xi|^\alpha-M^\alpha)\d \xi\\
&\leq -c(c_0, d, p)\int_{\mR^d}|(-\Delta)^{\alpha/4}|g|^{p/2}|^2\d x+{c_0M^\alpha}\|g\|_p^p.
\end{align*}
This in turn gives \eqref{Eq-Bernstein2} by \eqref{Eq-NewBernstein1} (taking $g=\Delta_j f$). Inequality \eqref{Eq-positive} is trivial, so we omit its proof here. 
\end{proof}

The following result is a refined version of \cite[Theorem 1.1]{zhao2021regularity} (with $\sigma=\mI$ therein). 
\begin{theorem}\label{Thm-Regularity}
Under the same assumptions of Theorem \ref{Thm}, there exists a constant $\lambda_0=\lambda_0(d, \alpha, c_0,M, \beta, \|b\|_\beta)$ such that for any $\lambda\geq \lambda_0$, equation \eqref{Eq-PDE} has a unique solution $u\in\bigcap_{\gamma<\beta} C^{\alpha+\gamma}$. Moreover, it holds that 
\begin{equation}
\lambda \|u\|_{C^\gamma}+ \|u\|_{C^{\alpha+\gamma}} \leq C \|f\|_{C^\beta}, 
\end{equation}
where $C$ only depends on $d,\alpha, c_0, M, \beta, \gamma$ and  $\|b\|_{C^\beta}$. 
\end{theorem}
\begin{proof}
For all $p\in [2,\infty)$, replacing  \cite[Lemma 3.1]{zhao2021regularity} by our Lemma \ref{Le-Bernstein2} and following the proof for \cite[Theorem 3.8]{zhao2021regularity}, one can see that 
\begin{equation}\label{eq-WC}
    \sup_{z\in \mR^d} \l(\lambda \|u\chi(\cdot+z)\|_{W^{\beta}_p}+\|u\chi(\cdot+z)\|_{W^{\alpha+\beta}_p} \r) \leq C \sup_{z\in \mR^d}\|f\chi(\cdot+z)\|_{W^\beta_p}.
\end{equation}
Due to \cite[Lemma 2.6]{zhao2021regularity}, for any $s>d/p$ and $\eps>0$, it holds that 
$$
\|u\|_{C^{s-d/p}} \leq C \sup_{z\in \mR^d} \|u\chi(\cdot+z)\|_{W^{s}_p}\leq C \|u\|_{C^{s+\eps}}. 
$$
Thus, for any $0<\gamma<\theta<\beta$ and $p=d/(\theta-\gamma)$, we obtain 
\begin{align*}
    \lambda \|u\|_{C^\gamma}+ \|u\|_{C^{\alpha+\gamma}} \leq& \sup_{z\in \mR^d} \l(\lambda \|u\chi(\cdot+z)\|_{W^{\theta}_p}+\|u\chi(\cdot+z)\|_{W^{\alpha+\theta}_p} \r) \\
    \overset{\eqref{eq-WC}}{\leq}& C \sup_{z\in \mR^d}\|f\chi(\cdot+z)\|_{W^\theta_p}\leq C \|f\|_{C^\beta}.
\end{align*}
\end{proof}

\section{Proof of main result}\label{Sec-main}
\begin{proof}[Proof of Theorem \ref{Thm}]
Fix $\gamma\in (1-\frac{\alpha}{2}, \beta)$. Based on our assumptions and Theorem \ref{Thm-Regularity}, we can conclude that equation \eqref{Eq-PDE} with $f=b$ has a unique solution $u\in C^{\alpha+\gamma}$. Furthermore, if $\theta :=\frac{\alpha+\gamma-1}{\alpha}>0$ and $\lambda$ is large enough, by interpolation, the following inequality holds:
\begin{equation}\label{eq-u}
    \|u\|_{C^1} \leq \|u\|_{C^\gamma}^\theta \|u\|_{C^{\alpha+\gamma}}^{1-\theta} \leq  C \lambda^{-\theta}\leq \frac{1}{4}.
\end{equation}
Let $N(\d r, \d z)$ be the Poisson random measure associated with $Z$, whose intensity measure is given by $\d r\,\nu(\d z)$. Then  
\[
    Z_t= \int_0^t\!\!\!\int_{|z|<1} z \widetilde{N}(\d r, \d z)+ \int_0^t\!\!\!\int_{|z|\geq 1} z {N}(\d r, \d z),  
\]
where $\widetilde{N}(\d r, \d z)=N(\d r, \d z)-\d r\nu(\d z)$. 
Thanks to the generalized It\^o's formula (see \cite{priola2012pathwise}), we have 
\begin{align*}
&u(X^n_t)-u(X^n_s)\\
=& \int_s^t \Big( Lu(X^n_r)+b(X^n_{k_n(r)})\cdot \nabla u(X^n_r) \Big)\,\d r   + \int_s^t\!\!\!\int_{\mR^d}\Big(u(X^n_{r-}+z)-u(X^n_{r-})\Big) \widetilde N(\d r, \d z )\\
=& \int_s^t  (\lambda u-b)(X^n_r) \d r +  \int_s^t  \Big(b(X^n_{k_n(r)}- b(X^n_r))\Big) \cdot \nabla u(X^n_r) \d r \\
&+ \int_s^t\!\!\!\int_{\mR^d}\Big(u(X^n_{r-}+z)-u(X^n_{r-})\Big) \widetilde N(\d r, \d z ),
\end{align*}
which implies 
\begin{equation}\label{eq-bn}
\begin{aligned}
\int_s^t b(X^n_r) \d r =& u(X^n_s)- u(X^n_t) + \lambda \int_s^t u(X^n_r) \d r +  \int_s^t  \Big(b(X^n_{k_n(r)}- b(X^n_r))\Big) \cdot \nabla u(X^n_r) \d r \\
&+ \int_s^t\!\!\!\int_{\mR^d}\Big(u(X^n_{r-}+z)-u(X^n_{r-})\Big) \widetilde N(\d r, \d z ). 
\end{aligned}
\end{equation}
Similarly, 
\begin{equation}\label{eq-b}
\begin{aligned}
\int_s^t b(X_r) \d r =& u(X_s)- u(X_t) + \lambda \int_s^t u(X_r) \d r \\
&+ \int_s^t\!\!\!\int_{\mR^d} \Big(u(X_{r-}+z)-u(X_{r-})\Big) \widetilde N(\d r, \d z ). 
\end{aligned}
\end{equation}
By definition, 
\begin{equation*}
\begin{aligned}
X^n_t-X_t=&X_0^n-X_0+\int_0^t \Big( b(X^n_{k_n(r)})-b(X^n_{r}) \Big) \d r + \int_0^t \Big(b(X^n_r)-b(X_r)\Big) \d r
\end{aligned}
\end{equation*}
Plugging \eqref{eq-bn} and \eqref{eq-b} in to the above equation, we obtain 
\begin{equation*}
\begin{aligned}
X^n_t-X_t=& X_0^n-X_0+ \int_0^t \Big( b(X^n_{k_n(r)})-b(X^n_{r}) \Big) \d r +\Big(u(X^n_0)-u(X_0) \Big) + \Big(u(X_t)-u(X^n_t) \Big) \\
&+ \lambda \int_0^t \Big( u(X^n_r)-u(X_r) \Big) \d r +  \int_0^t  \Big(b(X^n_{k_n(r)}- b(X^n_r))\Big) \cdot \nabla u(X^n_r) \d r \\
&+ \int_0^t\!\!\!\int_{\mR^d} \l[ \Big(u(X^n_{r-}+z)-u(X^n_{r-})\Big)-\Big(u(X_{r-}+z)-u(X_{r-})\Big)\r] \widetilde N(\d r, \d z ) \\
=&:\sum_{i=1}^7 I_i. 
\end{aligned}
\end{equation*}
For $I_2$ and $I_6$, by our assumption on $b$ and \eqref{eq-u}, we have 
$$
 \bE |I_2|^p\lesssim\int_0^t \bE \l(1\wedge | X^n_{k_n(r)}-X^n_r |^{\beta}\r)^p, \quad \bE |I_6|^p\lesssim\int_0^t \bE \l(1\wedge | X^n_{k_n(r)}-X^n_r |^{\beta}\r)^p. 
$$
For $I_3$, $I_4$ and $I_5$, again using \eqref{eq-u}, one sees that 
\[
|I_3| \leq \frac{1}{4}|X_0^n-X_0|,\quad |I_4| \leq \frac{1}{4}|X_t^n-X_t|, \quad I_5\leq C \lambda^{1-\theta} \int_0^t |X^n_r-X_r| \d r. 
\]
For $I_7$, by BDG inequatlity and the basic fact that $|f(x+z)-f(x)-f(y+z)+f(y)|\lesssim \|f\|_{C^{\alpha+\gamma}} |x-y| (1\wedge |z|^{\alpha+\gamma-1})$, we get 
\begin{equation*}
 \begin{aligned}
 \bE |I_7|^p \lesssim& \bE \l( \int_0^t\!\!\!\int_{\mR^d} |X^n_{r-}-X_{r-}|^2 (1\wedge |z|)^{2\alpha+2\gamma-2} \nu(\d z )\,\d r ) \r)^{p/2}\\
 \lesssim& \bE \l(\int_0^t |X^n_{r}-X_r|^2 \d r\r)^{p/2}. 
 \end{aligned}
\end{equation*}
Combining all the estimates above, we arrive 
\begin{equation*}
\begin{aligned}
\bE \sup_{t\in [0,T_0]} |X^n_t-X_t|^p \lesssim& \bE |X^n_0-X_0|^p + (T_0^p+T_0^{p/2})\bE \sup_{t\in [0,T_0]} |X^n_t-X_t|^p \\
&+ \int_0^{T_0} \bE(1\wedge |X^n_{k_n(t)}-X^n_t|^{\beta p}) \d t. 
\end{aligned}
\end{equation*}
Therefore, for $T_0\in (0,1]$ sufficiently small, we have 
\begin{align*}
\bE \sup_{t\in [0,T_0]} |X^n_t-X_t|^p \lesssim& \bE |X^n_0-X_0|^p+  \int_0^{T_0} \bE\l(1\wedge |Z_{k_n(t)}-Z_t|^{\beta p}\r) \d t\\
\lesssim & \bE |X^n_0-X_0|^p+  \bE ( 1\wedge |Z_{1/n}|^{\beta p} ).  
\end{align*}
Iterating this procedure finite times we reach the full time horizon $[0,1]$. 
\end{proof}

\begin{corollary}
   If $Z$ is a nondegenerate $\alpha$-stable process, then  
    \begin{align*}
        \bE \sup_{t\in [0,1]} |X_t^n-X_t|^p \leq C  \bE|X^n_0-X_0|^p+C \left\{
        \begin{aligned}
          & n^{-p\beta/\alpha},&  &\mbox{ if } 0<p<\alpha/\beta, \\
          &n^{-1}\log n,&  &\mbox{ if } p=\alpha/\beta, \\
          &n^{-1},&  &\mbox{ if } p>\alpha/\beta. 
        \end{aligned}
        \right.
    \end{align*} 
\end{corollary}

\begin{proof}
Since $Z_{t}\overset{d}{=} t^{1/\alpha}Z_1$, we have 
\begin{equation*}
    \begin{aligned}
    \bE \l(1\wedge |Z_{1/n}|^{p}\r) =& \bE \l(1\wedge n^{-\frac{p}{\alpha}}|Z_{1}|^{p}\r)\\
    =& \bP(|Z_1|> n^{\frac{1}{\alpha}}) + n^{-\frac{p}{\alpha}}\bE\l(|Z_1|^{p}\1_{\{|Z_1|\leq n^{1/\alpha}\}}\r)\\
    \lesssim& n^{-1} + n^{-\frac{p}{\alpha}} \int_0^{n^{1/\alpha}} t^{p-1} \bP (|Z_1|>t) \d t \\
    \lesssim & n^{-1} + n^{-\frac{p}{\alpha}}+ n^{-\frac{p}{\alpha}} \int_1^{n^{1/\alpha}} t^{p-\alpha-1} \d t \\
    \lesssim & \left\{
        \begin{aligned}
          & n^{-p/\alpha},&  &\mbox{ if } 0<p<\alpha, \\
          &n^{-1}\log n,&  &\mbox{ if } p=\alpha, \\
          &n^{-1},&  &\mbox{ if } p>\alpha. 
        \end{aligned}
        \right.
\end{aligned}
\end{equation*}
So we complete our proof. 
\end{proof}

\bibliographystyle{alpha}
\bibliography{mybib}

\end{document}